\documentclass[10pt,reqno,a4]{imsart}
\usepackage{a4}
\usepackage{graphicx}
\usepackage{amsmath,amssymb,amsthm,natbib,color,bbm,ifthen,graphicx,epsf,enumerate,bbm}
\usepackage{xspace}
\usepackage[latin1]{inputenc}

\newcommand{\Xset}{\ensuremath{\mathsf{X}}}
\newcommand{\Xsigma}{\mathcal{B}(\Xset)}
\newcommand{\Yset}{\ensuremath{\mathsf{Y}}}
\newcommand{\Ysigma}{\mathcal{B}{\Yset}}
\newcommand{\Zset}{\ensuremath{\mathsf{Z}}}

\newcommand{\Xproc}{\ensuremath{\{X_k\}_{k\geq 0}}}
\newcommand{\Uproc}{\ensuremath{\{U_k\}_{k\geq 0}}}
\newcommand{\Vproc}{\ensuremath{\{V_k\}_{k\geq 0}}}
\newcommand{\XYproc}{\ensuremath{\{X_k,Y_k\}_{k\geq 0}}}
\newcommand{\Yproc}{\ensuremath{\{Y_k\}_{k\geq 0}}}

\newcommand{\forvar}[2][]%
{%
\ifthenelse{\equal{#1}{}}{\ensuremath{\alpha_{#2}}}{\ensuremath{\alpha_{#1,#2}}}%
}
\newcommand{\mcb}[1]{\ensuremath{\mathcal{F}_{\mathrm{b}}\left(#1\right)}}

\newcommand{\diam}{\mathrm{diam}}
\newcommand{\as}{a.s.}

\newcommand{\backvar}[2]{\ensuremath{\beta_{#1|#2}}}

\newcommand{\PP}{\ensuremath{\operatorname{P}}}
\newcommand{\PE}{\ensuremath{\operatorname{E}}}
\newcommand{\CPE}[3][]
{\ifthenelse{\equal{#1}{}}{\operatorname{E}\left[\left. #2 \, \right| #3 \right]}{\operatorname{E}_{#1}\left[\left. #2 \, \right | #3 \right]}}
\newcommand{\CPEs}[3][]
{\ifthenelse{\equal{#1}{}}{\operatorname{E}^\star\left[\left. #2 \, \right| #3 \right]}{\operatorname{E}^\star_{#1}\left[\left. #2 \, \right | #3 \right]}}

\newcommand{\CPP}[3][]
{\ifthenelse{\equal{#1}{}}{\operatorname{P}\left(\left. #2 \, \right| #3 \right)}{\operatorname{P}_{#1}\left(\left. #2 \, \right | #3 \right)}}
\newcommand{\CPPs}[3][]
{\ifthenelse{\equal{#1}{}}{\operatorname{P}^\star \left(\left. #2 \, \right| #3 \right)}{\operatorname{P}^\star_{#1}\left(\left. #2 \, \right | #3 \right)}}

\newcommand{\cpp}[3][]
{\ifthenelse{\equal{#1}{}}{\operatorname{p}\left(\left. #2 \, \right| #3 \right)}{\operatorname{p}_{#1}\left(\left. #2 \, \right | #3 \right)}}

\newcommand{\PVar}{\ensuremath{\operatorname{Var}}}
\newcommand{\PCov}{\ensuremath{\operatorname{Cov}}}

\newcommand{\rmd}{\ensuremath{\mathrm{d}}}

\newcommand{\rme}{\ensuremath{\mathrm{e}}}
\newcommand{\tvnorm}[1]{\ensuremath{\left\|#1\right\|_{\mathrm{TV}}}}

\newcommand{\Nset}{\mathbb{N}}

\newcommand{\rset}{\mathbb{R}}

\renewcommand{\epsilon}{\varepsilon}

\newcommand{\chunk}[4][]%
{\ifthenelse{\equal{#1}{}}{\ensuremath{{#2}_{#3:#4}}}{\ensuremath{#2^#1}_{#3:#4}}
}

\newcommand{\eqsp}{\;}
\newcommand{\eqdef}{\ensuremath{\stackrel{\mathrm{def}}{=}}}
\newcommand{\1}{\ensuremath{\mathbbm{1}}}

\newcommand\W[3]{W[#1](#2;#3)}

\newcommand{\filt}[2][]%
{%
\ifthenelse{\equal{#1}{}}{\ensuremath{\phi_{#2}}}{\ensuremath{\phi_{#1,#2}}}%
}
\newcommand{\pred}[3][]%
{%
\ifthenelse{\equal{#1}{}}{\ensuremath{\phi_{#2|#3}}}{\ensuremath{\phi_{#1,#2|#3}}}%
}
\newcommand{\post}[3][]%
{%
\ifthenelse{\equal{#1}{}}{\ensuremath{\phi_{#2|#3}}}{\ensuremath{\phi_{#1,#2|#3}}}%
}

\newcommand{\FK}[2]{\ensuremath{\mathrm{F}_{#1|#2}}}
\newcommand{\fk}[2]{\ensuremath{\mathrm{f}_{#1|#2}}}
\newcommand{\bFK}[2]{\ensuremath{\bar{\mathrm{F}}_{#1|#2}}}
\newcommand{\RFK}[2]{\ensuremath{\bar{\mathrm{R}}_{#1|#2}}}
\newcommand{\CS}[2]{\ensuremath{\bar{C}_{#1|#2}}}

\def\marg#1{}
\def\boxit#1#2#3{\hbox{\vrule width #2               
                     \vtop{%
                             \vbox{\hrule  height #2 \kern#1 
                           \hbox{\kern#1 #3\kern#1}  
                          }%
                     \kern#1 \hrule  height #2   
                  }%
                         \vrule width #2             
                      } %
                  }
\newcommand{\Xinit}{\ensuremath{\xi}}

\def\eps{\ensuremath{\epsilon}}

\def\iid{i.i.d.\ }
\def\al{\ensuremath{\alpha}}
\def\en{\infty}

\newlength{\deno}  
\newlength{\nal}   \newlength{\naal}

\def\lm{\lambda}

\newcommand{\convs}[1]{\stackrel{#1}{\rightarrow}}
\newcommand{\convl}[1]{\stackrel{#1}{\longrightarrow}}
\newcommand{\convll}[2][\null]{
  \@ifundefined{bigaw@}{\newdimen\bigaw@}{}
    \setboxz@h{$\m@th\scriptstyle\;{#2}\;\;$}%
  \ifdim\wdz@>\bigaw@ pt\global\bigaw@\wdz@\fi
  \@ifnotempty{#1}{\setbox\@ne\hbox{$\m@th\scriptstyle\;{#1}\;\;$}%
    \ifdim\wd\@ne>\bigaw@\global\bigaw@\wd\@ne\fi}%
    \mathrel{\mathop{\hbox to\bigaw@{\rightarrowfill@\displaystyle}}%
     \limits^{#2}\@ifnotempty{#1}{_{#1}}}%
    }

\newcommand{\conv}[2][abcdefg]{%
    \setbox\ewa\hbox{$#1$}     \setbox\ewb\hbox{$#2$}%
    \setlength{\deno}{\wd\ewb}
    \ifthenelse{\equal{#1}{abcdefg}}
        {\setbox\ert\hbox{$a$} \setbox\prt\hbox{$aa$}%
        \setlength{\nal}{\wd\ert}
        \setlength{\naal}{\wd\prt}
        \ifthenelse{\lengthtest{\deno<\nal}}%
            {\convs{#2}}%
            {\ifthenelse{\lengthtest{\deno>\naal}}%
                {\convll{#2}}%
                {\convl{#2}}%
            }
        }%
    {\convll[#1]{#2}}%
    }

\def\cas{\convl{\rm a.s.}{}}

\newcommand{\eqsplit}[2][*]%
  {\ifthenelse{\equal{#1}{*} \or\equal{#1}{<???>}
                     \or\equal{#1}{nn} \or\equal{#1}{ } \or\equal{#1}{}}
    { \begin{align*}%
         #2 %
         \end{align*}%
        }
    {\begin{equation}\label{#1}\begin{split}\allowdisplaybreaks%
         #2%
         \end{split}\end{equation}
        }
  }

\newcounter{hyp}

\def\mcf{\mathcal{F}}

\def\ie{\textit{i.e., }}
\newcommand{\eg}{\textit{e.g.,}}
\newcommand{\wrt}{with respect to}

\newtheorem{thm}{Theorem}
\newtheorem{lemma}[thm]{Lemma}
\newtheorem{defi}{Definition}
\newtheorem{prop}[thm]{Proposition}
\newtheorem{cor}[thm]{Corollary}
\newtheorem{example}{Example}
\theoremstyle{remark}
\newtheorem{remark}{Remark}

\begin{document}
\begin{frontmatter}
\title{Forgetting of the initial condition for the filter in general state-space hidden Markov chain: a coupling approach}

\begin{aug}
\author{Randal Douc \ead[label=e1]{douc@cmapx.polytechnique.fr}}
\address{GET/Télécom INT, \\
France, \\ \printead{e1}}

\author{Eric Moulines \ead[label=e2]{moulines@tsi.enst.fr}}
\address{GET/Télécom Paris, \\ 46 rue Barrault,
75634 Paris Cédex 13, France, \\ \printead{e2}}

\author{Ya'acov Ritov \ead[label=e3]{yaacov.ritov@huji.ac.il}}
\address{Department of Statistics, The Hebrew University of Jerusalem, \\ \printead{e3}}

\end{aug}
\end{frontmatter}

\begin{abstract}
We give simple conditions that ensure exponential forgetting of the initial conditions of the filter
for general state-space hidden Markov chain.
The proofs are based on the coupling argument applied to the posterior Markov kernels.
These results are useful both for filtering hidden Markov models using approximation methods (e.g., particle filters)
and for proving asymptotic properties of estimators.
The results are general enough to cover models \emph{like} the Gaussian state space model,
without using the special structure that permits the application of the Kalman filter.
\end{abstract}

\maketitle
\section{Introduction and Notation}
We consider the filtering problem for a Markov chain $\XYproc$ with \emph{state} $X$ and
\emph{observation} $Y$.
The state process  $\Xproc$ is an homogeneous Markov chain taking value
in a measurable set $\Xset$ equipped with a $\sigma$-algebra $\Xsigma$. We let $Q$ be the transition kernel of the
chain.   The observations  $\Yproc$ takes
values in a measurable set $\Yset$ ($\Ysigma$ is the associated $\sigma$-algebra). For $i\leq j$, denote
$\chunk{Y}{i}{j} \triangleq(Y_i,Y_{i+1},\cdots,Y_j)$. Similar notation will
be used for other sequences. We assume furthermore. that  for each $k\geq 1$ and given $X_k$, $Y_k$ is independent of
$\chunk{X}{1}{k-1}$,$\chunk{X}{k+1}{\infty}$, $\chunk{Y}{1}{k-1}$, and
$\chunk{Y}{k+1}{\infty}$. We also assume that for each $x \in \Xset$, the conditional law has a density $g(x,\cdot)$
with respect to some fixed $\sigma$-finite measure  on the Borel $\sigma$-field $\mathcal{B(Y)}$.

We denote by $\filt[\Xinit]{n}[\chunk{y}{0}{n}]$
the distribution of the hidden state $X_n$ conditionally on the observations
$\chunk{y}{0}{n} \eqdef [y_0, \dots, y_n]$, which is given by
\begin{equation}
  \label{eq:filtering-distribution-1}
  \filt[\Xinit]{n}[\chunk{y}{0}{n}](A)  = \frac{\int_{\Xset^{n+1}} \Xinit(dx_0) g(x_0,y_0) \prod_{i=1}^n
    Q(x_{i-1},dx_i) g(x_i,y_i) \1_A(x_n)}{\int_{\Xset^{n+1}} \Xinit(dx_0)
    g(x_0,y_0) \prod_{i=1}^n Q(x_{i-1},dx_i) g(x_i,y_i)} \eqsp,
\end{equation}

In practice the model is rarely known exactly and therefore suboptimal filters are
computed by replacing the unknown transition kernel, likelihood function and
initial distribution by approximations.

The choice of these quantities plays a key role both when studying the
convergence of sequential Monte Carlo methods or when analysing the asymptotic
behaviour of the maximum likelihood estimator (see \eg\ \cite{delmoral:2004} or
\cite{cappe:moulines:ryden:2005} and the references therein).
A key point when analyzing maximum likelihood estimator or the  stability
of the filter over infinite horizon is to ask whether $\filt[\Xinit]{n}[\chunk{y}{0}{n}]$ and
$\filt[\Xinit']{n}[\chunk{y}{0}{n}]$ are close (in some sense) for large values
of $n$, and two different choices of the initial distribution $\Xinit$ and
$\Xinit'$.

The forgetting property of the initial condition of the optimal filter in
nonlinear state space models has attracted many research efforts and it is impossible
to give credit to every contributors. The purpose of the
short presentation of the existing results below is mainly to allow comparison
of assumptions and results presented in this contributions \wrt\ those
previously reported in the literature. The first result in this direction has
been obtained by \cite{ocone:pardoux:1996}, who established $L_p$-type
convergence of the optimal filter initialised with the wrong initial condition
to the filter initialised with the true initial distribution; their proof does not provide  rate of
convergence.  A new approach based on the Hilbert projective metric has later
been introduced in \cite{atar:zeitouni:1997} to establish the exponential
stability of the optimal filter \wrt\ its initial condition. However their
results are based on stringent \emph{mixing} conditions for the transition
kernels; these conditions state that there exist positive constants
$\epsilon_-$ and $\epsilon_+$ and a probability measure $\lambda$ on
$(\Xset,\Xsigma)$ such that for $f \in \mathbb{B}_+(\Xset)$,
\begin{equation}
\label{eq:mixing-condition}
\epsilon_- \lambda(f) \leq Q(x,f) \leq \epsilon_+ \lambda(f) \eqsp, \quad \text{for any $x \in \Xset$} \eqsp.
\end{equation}
This condition implies in particular  that the chain is uniformly geometrically
ergodic.  Similar results were obtained independently by
\cite{delmoral:guionnet:1998} using the Dobrushin ergodicity coefficient (see
\cite{delmoral:ledoux:miclo:2003} for further refinements of this result).
The mixing condition has later been weakened by
\cite{chigansky:lipster:2004}, under the assumption that the kernel $Q$ is
positive recurrent and is dominated by some reference measure $\lambda$:
\[
\sup_{(x,x') \in \Xset \times \Xset} q(x,x') < \infty \quad \text{and} \quad \int \mathrm{ess inf} q(x,x') \pi(x) \lambda(dx) > 0 \eqsp,
\]
where $q(x,\cdot)= \frac{d Q(x,\cdot)}{d\lambda}$, $\mathrm{ess inf}$ is the essential infimum \wrt\ $\lambda$ and
$\pi d \lambda$ is the stationary distribution of the chain $Q$ . Although the upper
bound is reasonable, the lower bound is restrictive in many applications and
fails to be satisfied \eg\ for the linear state space Gaussian model.

In \cite{legland:oudjane:2003}, the stability of the optimal filter is studied
for a class of kernels referred to as \emph{pseudo-mixing}. The definition of
pseudo-mixing kernel is adapted to the case where the state space is $\Xset=
\rset^d$, equipped with the Borel sigma-field $\Xsigma$.  A kernel $Q$ on
$(\Xset,\Xsigma)$ is \emph{pseudo-mixing} if for any compact set $C$ with a
diameter $d$ large enough, there exist positive constants $\epsilon_-(d) >0$
and $\epsilon_+(d) > 0$ and a  measure $\lambda_C$ (which may
be chosen to be finite without loss of generality) such that
\begin{equation}
\label{eq:pseudo-mixing-kernel}
\epsilon_-(d) \lambda_C(A)\leq Q(x,A)\leq \epsilon_+(d) \lambda_C(A) \eqsp, \quad \text{for any $x \in C$, $A \in \Xsigma$}
\end{equation}
This condition implies that for any $(x',x'') \in C \times C$,
$$
\frac{\epsilon_-(d)}{\epsilon_+(d)} < \mathrm{essinf}_{x \in \Xset} q(x',x)/q(x'',x)\leq  \mathrm{esssup}_{x \in \Xset} q(x',x)/q(x'',x) \leq \frac{\epsilon_+(d)}{\epsilon_-(d)} \eqsp,
$$
where $q(x,\cdot) \eqdef d Q(x,\cdot)/ d \lambda_C$, and $\mathrm{esssup}$ and $\mathrm{essinf}$
denote the essential supremum and infimum \wrt\ $\lambda_C$.  This
condition is obviously more general than \eqref{eq:mixing-condition}, but still
it is not satisfied in the linear Gaussian case (see \cite[Example
4.3]{legland:oudjane:2003}).

Several attempts have been made to establish the stability conditions under the
so-called \emph{small} noise condition. The first result in this direction has
been obtained by \cite{atar:zeitouni:1997} (in continuous time) who considered
an ergodic diffusion process with constant diffusion coefficient and linear
observations: when the variance of the observation noise is sufficiently small,
\cite{atar:zeitouni:1997} established that the filter is exponentially stable.
Small noise conditions also appeared (in a discrete time setting) in
\cite{budhiraja:ocone:1999} and \cite{oudjane:rubenthaler:2005}. These results
do not allow to consider the linear Gaussian state space model with arbitrary
noise variance.

More recently, \cite{chigansky:lipster:2006}  prove that the nonlinear filter forgets its initial condition in mean over the observations
for functions satisfying some integrability conditions.
The main result presented in this paper relies on the martingale convergence theorem rather than
direct analysis of filtering equations. Unfortunately, this method of proof cannot provide any rate
of convergence.

It is tempting to assume that forgetting of the initial condition should be true in general,
and that the lack of proofs for the general state-space case  is only a matter of technicalities.
The heuristic argument says that either
\begin{itemize}
\item the observations $Y$'s are informative, and we learn about the hidden state $X$
from the $Y$s around it, and forget the initial starting point.
\item the observations $Y$s are non-informative, and then the $X$ chain is moving by itself, and by itself it forgets its initial
condition, for example if it is positive recurrent.
\end{itemize}
Since we expect that the forgetting of the initial condition is retained in these two extreme cases,
it is probably so under any condition. However, this argument is false, as is shown by the following examples where the conditional chain does not
forget its initial condition whereas the unconditional chain does.
On the other hand, it can be that observed process, $\Yproc$ is not ergodic, while the conditional chain uniformly forgets the initial condition.
\begin{example}
Suppose that $\Xproc$ are \iid\ $B(1,1/2)$. Suppose $Y_i=\1(X_i =
X_{i-1})$. Then $\CPP{X_i=1}{X_0=0,\chunk{Y}{0}{n}}=1-\CPP{X_i=1}{X_1=1,\chunk{Y}{0}{n}} \in \{0,1\}$.

Here is a slightly less extreme example. Consider a Markov chain on the unit circle. All values below
are considered modulus $2\pi$. We assume that $X_i=X_{i-1}+U_i$,
where the state noise $\{U_k\}_{k \geq 0}$ are \iid. The
chain is hidden by additive white noise: $Y_i=X_i+\eps_i$, $\eps_i=\pi W_i + V_i$, where $W_i$ is Bernoulli random variable independent of $V_i$.
Suppose now that $U_i$ and $V_i$ are symmetric and supported on some small interval.
The hidden chain does not forget its initial distribution under this model. In fact the support of the
distribution of $X_i$ given $\chunk{Y}{0}{n}$ and $X_0=x_0$ is disjoint
from the support of its distribution given $\chunk{Y}{0}{n}$ and
$X_0=x_0+\pi$.

On the other hand, let $\Yproc$ be an arbitrary process.
Suppose it is modeled (incorrectly!) by a  autoregressive process observed in additive noise.
We will show that under different assumptions on the distribution of the state and the observation noise,
the conditional chain (given the observations $Y$s which are not necessarily generated by the model)
forgets its initial condition geometrically fast.
\end{example}

The proofs presented in this paper  are based on generalization of the notion of small sets and coupling of the two (non-homogenous)
Markov chains sampled from the distribution of $\chunk{X}{0}{n}$ given $\chunk{Y}{0}{n}$.
The coupling argument is based on presenting two chains $\{X_k\}$ and $\{X'_k\}$, which marginally follow the same sequence of transition kernels,
but have different initial distributions of the starting state. The chains move independently, until they \emph{coupled} at a random time $T$,
and from that time on they remain equal.

Roughly speaking, the two copies of the chain may couple at a time $k$
if they stand close one to the other. Formally, we mean by that, that the the pair of states of the two chains at time $k$ belong to some set,
which may depend of the current, but also past and future observations.
The novelty of the current paper is by considering sets which are in fact genuinely defined by the pair of states.
For example, the set can be defined as $\{(x,x'):\;\|x-x'\|<c\}$. That is, close in the usual sense of the word.

The prototypical example we use is the non-linear state space model:
 \eqsplit[eq:stateSpaceModel]{
    X_i&=a(X_{i-1}) + U_i
    \\
    Y_i&= b(X_i) + V_i ,
  }
where $\Uproc$ is the \emph{state noise} and $\Vproc$ is the \emph{measurement noise}. Both $\Uproc$ and $\Vproc$ are assumed to be \iid\ and
 mutually independent. Of course, the filtering problem for the linear version of this model with independent Gaussian noise
is solved explicitly by the Kalman filter. But this is one of the few non-trivial models which admits a simple solution.
Under the Gaussian linear model, we argue that whatever are $\chunk{Y}{0}{n}$, two independent chains drawn from the conditional distribution
will be remain close to each other even if the $Y$s are drifting away. Any time they will be close, they will be able to couple, and this will
happen quite frequently.

Our approach for proving that a chain forgets its initial conditions can be decomposed in two stages.
We first argue that there are \emph{coupling sets} (which may depend on the observations, and may also vary according to the iteration index)
where we can couple two copies of the chains, drawn independently from the conditional distribution given the observations and started
from two different initial conditions,  with a probability which is an explicit function of the observations.
We then argue that a pair of chains are likely to drift frequently towards these coupling sets.

The first group of results identify situations in which the coupling set is given in a product form,
and in particular in situations where $\Xset \times \Xset$ is a coupling set.
In the typical situation,  many values of $Y_i$ entail that  $X_i$ is in some set with high probability,
and hence  the two conditionally independent copies are likely to be in this set and close to each other.
In particular, this enables us to prove the convergence of (nonlinear) state space processes with bounded noise and,
more generally, in situations where the tails of the observations error is thinner than those of dynamics innovations.

The second argument generalizes the standard drift condition to the coupling set.
The general argument specialized to the linear-Gaussian state model is surprisingly simple.
We generalize this argument to the linear model where both the dynamics innovations and the measurement errors have strongly unimodal density.

\section{Notations and definitions}
Let $n$ be a given positive index and consider the finite-dimensional
distributions of $\Xproc$ given $\chunk{Y}{0}{n}$. It is well known (see \cite[Chapter 3]{cappe:moulines:ryden:2005})
that, for any positive index $k$, the distribution of $X_k$ given $\chunk{X}{0}{k-1}$ and $\chunk{Y}{0}{n}$
reduces to that of $X_k$ given $X_{k-1}$ only and $\chunk{Y}{0}{n}$. The following definitions will be instrumental in
decomposing the joint posterior distributions.

\begin{defi}[Backward functions]
For $k \in \{0, \dots, n\}$, the backward function
$\backvar{k}{n}$ is the non-negative measurable function on $\Yset^{n-k} \times \Xset$ defined by
\begin{multline}
  \backvar{k}{n}(\chunk{y}{k+1}{n},x) = \\
  \idotsint Q(x,dx_{k+1}) g(x_{k+1},y_{k+1}) \prod_{l={k+2}}^n
  Q(x_{l-1},dx_l) g(x_l,y_l) \eqsp ,
\label{eq:def:beta:kernel}
\end{multline}
for $k \leq n-1$ (with the same convention that the rightmost product is empty
for $k=n-1$); $\backvar{n}{n}(\cdot)$ is set to the constant function equal
to 1 on $\Xset$.
\end{defi}
The term ``backward variables''  is part of the HMM credo and dates back to the seminal
work of Baum and his colleagues~\cite[p.~168]{baum:petrie:soules:weiss:1970}.
The backward functions  may be obtained, for all $x \in\Xset$ by the recursion
\begin{equation}
\backvar{k}{n}(x) = \int Q(x,dx') g(x',y_{k+1}) \backvar{k+1}{n}(x') \eqsp
\label{eq:rec:beta}
\end{equation}
operating on decreasing indices $k = n-1$ down to $0$ from the initial
condition
\begin{equation}
\backvar{n}{n}(x) = 1 \eqsp.
\label{eq:def:beta:init}
\end{equation}
\begin{defi}[Forward Smoothing Kernels]
  \label{defi:forward_kernel}
  Given $n\geq 0$, define for indices $k \in \{0, \dotsc, n-1\}$ the transition
kernels
\begin{equation}
\FK{k}{n}(x,A) \eqdef
\begin{cases}
[\backvar{k}{n}(x)]^{-1} \int_A Q(x,dx') g(x',y_{k+1}) \backvar{k+1}{n}(x') & \text{if $\backvar{k}{n}(x) \neq 0$} \\
0 & \text{otherwise} \eqsp ,
\end{cases}
\label{eq:smoothing:forward_kernel}
\end{equation}
for any point $x \in \Xset$ and set $A \in \Xsigma$. For indices $k \geq n$,
simply set
\begin{equation}
\FK{k}{n} \eqdef Q \eqsp ,
\label{eq:smoothing:forward:kernel:after}
\end{equation}
where $Q$ is the transition kernel of the unobservable chain $\Xproc$.
\end{defi}

Note that for indices $k \leq n-1$, $\FK{k}{n}$ depends on the future
observations $\chunk{Y}{k+1}{n}$ through the backward variables
$\backvar{k}{n}$ and $\backvar{k+1}{n}$ only.
The subscript $n$ in the $\FK{k}{n}$
notation is meant to underline the fact that, like the backward functions
$\backvar{k}{n}$, the forward smoothing kernels $\FK{k}{n}$ depend on the final
index $n$ where the observation sequence ends.
Thus, for any $x \in \Xset$, $A \mapsto \FK{k}{n}(x,A)$ is a probability
measure on $\Xsigma$. Because the functions $x \mapsto \backvar{k}{n}(x)$ are
measurable on $(\Xset,\Xsigma)$, for any set $A \in \Xsigma$, $x \mapsto
\FK{k}{n}(x,A)$ is $\Xsigma/\mathcal{B}(\rset)$-measurable.  Therefore,
$\FK{k}{n}$ is indeed a Markov transition kernel on $(\Xset,\Xsigma)$.

Given $n$, for any index $k \geq 0$ and function $f \in \mcb{\Xset}$,
\begin{equation*}
\PE_\Xinit [ f(X_{k+1}) \mid \chunk{X}{0}{k},\chunk{Y}{0}{n} ] = \FK{k}{n}(X_k,f) \eqsp.
\end{equation*}
More generally, For any integers $n$ and $m$, function $f \in \mcb{\Xset^{m+1}}$ and initial probability $\Xinit$ on $(\Xset,\Xsigma)$,
\begin{equation}
\PE_\Xinit [ f( \chunk{X}{0}{m}) \mid \chunk{Y}{0}{n} ] =
 \idotsint   f(\chunk{x}{0}{m}) \: \post[\Xinit]{0}{n}(dx_0) \prod_{i=1}^m \FK{i-1}{n}(x_{i-1},dx_i) \eqsp ,
\label{eq:smoothing:forw_decomp}
\end{equation}
where $\{ \FK{k}{n}\}_{k \geq 0}$ are defined
by~\eqref{eq:smoothing:forward_kernel} and
\eqref{eq:smoothing:forward:kernel:after} and $\post[\Xinit]{k}{n}$ is the
marginal smoothing distribution of the state $X_k$ given the observations $\chunk{Y}{0}{n}$. Note that $\post[\Xinit]{k}{n}$ may be expressed, for any $A \in \Xsigma$, as
\begin{equation}
   \post[\Xinit]{k}{n}(A) = \left[\int \filt[\Xinit]{k}(dx) \backvar{k}{n}(x)\right]^{-1} \int_A \filt[\Xinit]{k}(dx)  \backvar{k}{n}(x)  \eqsp,
\label{eq:smoothing:marginal-posterior}
\end{equation}
where $\filt[\Xinit]{k}$ is the filtering distribution defined in \eqref{eq:filtering-distribution-1} and $\backvar{k}{n}$ is the backward function.

\section{The coupling construction and coupling sets}
\subsection{Coupling constant and the coupling construction}
As outlined in the introduction, our proofs are based on  coupling two copies of the conditional chain started from two different initial conditions.
For any two probability measures $\mu_1$ and $\mu_2$ we define
the total variation distance
$\tvnorm{\mu_1-\mu_2}=\sup_A|\mu_1(A)-\mu_2(A)|$ and we
also recall the identities $\sup_{|f|\leq 1}|\mu(f)|=2 \tvnorm{\mu_1-\mu_2}$ and
$\sup_{0\leq f\leq 1}|\mu(f)|= \tvnorm{\mu_1-\mu_2}$.
Let $n$ and $m$  be integers, and let $k \in \{0, \dots, n-m\}$.
Define the $m$-skeleton of the forward smoothing kernel as follows:
\begin{equation}
\label{eq:definition-forward-smoothing-product}
\FK{k,m}{n} \eqdef \FK{km+m-1}{n}\dots\FK{km}{n} \eqsp,
\end{equation}

\begin{defi}[Coupling constant of a set]
\label{def:coupling-constant}
Let $n$ and $m$ be integers, and let $k \in \{0, \dots, n-m\}$. The \emph{coupling constant} of the set $C \subset \Xset \times \Xset$
is  defined as
\begin{equation}
\label{eq:definition-coupling-constant}
\eps_{k,m|n}(C)\eqdef 1 - \frac{1}{2} \sup_{(x,x') \in C} \tvnorm{\FK{k,m}{n}(x,\cdot)-\FK{k,m}{n}(x',\cdot)} \eqsp.
\end{equation}
\end{defi}
The definition of the coupling constant implies that, for any $(x,x') \in C$,
\begin{equation}
\label{eq:coupling-set-definition}
\FK{k,m}{n}(x,A)\wedge \FK{k,m}{n}(x',A)\geq \eps_{k,m|n}(C) \nu^C_{k,m|n}(x,x';A) \eqsp.
\end{equation}
where
\begin{equation}
\label{eq:minorizing-probability}
\nu_{k,m|n}^C(x,x';A) = \frac{(\FK{k,m}{n}(x,\cdot) \wedge \FK{k,m}{n}(x',\cdot))(A)}{(\FK{k,m}{n}(x,\cdot) \wedge \FK{k,m}{n}(x',\cdot))(\Xset)},
\end{equation}
where for any measures $\mu$ and $\nu$ on $(\Xset,\Xsigma)$, $\mu \wedge \nu$ is the largest measure for which $(\mu \wedge \nu)(A) \geq \min(\mu(A),\nu(A))$, for all $A \in \Xsigma$.

We may now proceed to the coupling construction. Let $n$ be an integer, and for any $k \in \{0,\dots,\lfloor n / m \rfloor \}$,
let $\CS{k}{n}$ be a set-valued function, $\CS{k}{n}: \Yset^n \to \Xsigma \otimes \Xsigma$, where $\Xsigma \otimes \Xsigma$ is
the smallest $\sigma$-algebra containing the sets $A \times B$ with $A,B \in \Xsigma$.
We define $\RFK{k}{n}$ as the Markov transition kernel satisfying, for all $(x,x') \in \CS{k}{n}$ and for
all $A \in \Xsigma$ and $(x,x') \in \CS{k}{n}$,
\begin{multline}
\label{eq:definition-residual}
\RFK{k,m}{n}(x,x';A \times A') = \left\{ (1-\eps_{k,m|n})^{-1} (\FK{k,m}{n}(x,A) - \eps_{k,m|n} \nu_{k,m|n}(x,x';A)) \right\} \\ \times
\left\{ (1-\eps_{k,m|n})^{-1} (\FK{k,m}{n}(x',A') - \eps_{k,m|n} \nu_{k,m|n}(x,x';A')) \right\} \eqsp,
\end{multline}
where we have omitted the dependence upon the set $\CS{k}{n}$ in the definition of the coupling constant $\eps_{k,m|n}$ and of
the minorizing probability $\nu_{k,n|m}$.
For all $(x,x') \not \in \Xset \times \Xset$, we define
\begin{equation}
\label{eq:definition-product-kernel}
\bFK{k,m}{n}(x,x';\cdot)= \FK{k,m}{n}\otimes \FK{k,m}{n}(x,x';\cdot) \eqsp,
\end{equation}
where, for two kernels $K$ and $L$ on $\Xset$,   $K \otimes L$  is the tensor product of the kernels $K$ and $L$, \ie\
for all $(x,x') \in \Xset \times \Xset$ and $A,A' \in \mathcal{B}(\Xset)$
\begin{equation}
\label{eq:tensor-product-kernel}
K \otimes L(x,x';A\times A')= K(x,A) L(x',A') \eqsp.
\end{equation}
Define the product space $\Zset =\Xset \times \Xset \times \{ 0, 1\}$,
and the associated product sigma-algebra $\mathcal{B}(\Zset)$.  Define on the space
$(\Zset^\Nset, \mathcal{B}(\Zset)^{\otimes \Nset})$ a Markov chain $Z_i \eqdef (\tilde{X}_i, \tilde{X}'_i, d_i)$, $i \in \{0, \dots, n\}$ as follows.
If $d_i = 1$, then draw $\tilde{X}_{i+1} \sim \FK{i,m}{n}(\tilde{X}_i,\cdot)$, and set $\tilde{X}'_{i+1}=\tilde{X}_{i+1}$ and $d_{i+1} = 1$. Otherwise, if
$(\tilde{X}_i,\tilde{X}'_i) \in \CS{i}{n}$, flip a coin with probability of heads
$\eps_{i,m|n}$. If the coin comes up head, then draw $\tilde{X}_{i+1}$ from  $\nu_{i,m|n}(\tilde{X}_i,\tilde{X}'_i;\cdot)$, and set
$\tilde{X}'_{i+1} = \tilde{X}_{i+1}$ and $d_{i+1}=1$.
If the coin comes up tail, then draw $(\tilde{X}_{i+1},\tilde{X}'_{i+1})$  from the residual kernel $\RFK{i,m}{n}(\tilde{X}_i,\tilde{X}'_i; \cdot)$
and set $d_{i+1}=0$. If $(\tilde{X}_i,\tilde{X}'_i) \not\in \CS{i}{n}$,
then draw $(\tilde{X}_{i+1},\tilde{X}'_{i+1})$ according to the kernel $\bFK{i,m}{n}(\tilde{X}_i,\tilde{X}'_i; \cdot)$ and set $d_{i+1}= 0$.
For $\mu$ a probability measure on $\mathcal{B}(\Zset)$, denote $\PP^Y_{\mu}$
the probability measure induced by the Markov chain $Z_i$, $i \in \{0, \dots, n\}$ with initial
distribution $\mu$. It is then easily checked that
for any $i \in \{0, \dots, \lfloor n/ m \rfloor \}$ and any initial distributions $\Xinit$ and $\Xinit'$, and any $A,A' \in \mathcal{B}(\Xset)$,
\begin{align*}
\PP^Y_{\Xinit \otimes \Xinit' \otimes \delta_0} \left(Z_i \in A \times \Xset \times \{ 0,1 \} \right) &= \post[\Xinit]{im}{n}(A) \eqsp, \\
\PP^Y_{\Xinit \otimes \Xinit' \otimes \delta_0} \left(Z_i \in \Xset \times A' \times \{0,1\} \right)  &= \post[\Xinit']{im}{n}(A) \eqsp,
\end{align*}
where $\delta_x$ is the Dirac measure and $\otimes$ is the tensor product of measures and $\post[\Xinit]{k}{n}$ is 
the marginal posterior distribution given by \eqref{eq:smoothing:marginal-posterior}

Note that $d_i$ is the  \textit{bell variable}, which shall indicate
whether the chains have coupled ($d_i=1$) or not ($d_i=0$) by time $i$.
Define the \textit{coupling time}
\begin{equation}
\label{eq:coupling-time}
T = \inf \{ k \geq 1, d_k = 1 \} \eqsp,
\end{equation}
with the convention $\inf \emptyset = \infty$.
By the Lindvall inequality, the total variation distance
between the filtering distribution associated to two different
initial distribution $\Xinit$ and $\Xinit'$, \ie\ $\CPP[\Xinit]{X_n \in \cdot}{\chunk{Y}{0}{n}}$ and $\CPP[\Xinit']{X_n \in \cdot}{\chunk{Y}{0}{n}}$,
is bounded by the tail distribution of the coupling time,
\begin{equation}
\label{eq:lindvall-inequality}
\tvnorm{\CPP[\Xinit]{X_n \in \cdot}{\chunk{Y}{0}{n}}- \CPP[\Xinit']{X_n \in \cdot}{\chunk{Y}{0}{n}}}  \leq \PP^Y_{\Xinit \otimes \Xinit' \otimes \delta_0}(T \geq \lfloor n/m \rfloor) \eqsp.
\end{equation}
In the following section, we consider several conditions allowing to bound the tail distribution of the coupling time.

\subsection{Coupling sets}
Of course, the construction above is of interest only if we may find set-valued function $\CS{k}{n}$ such whose coupling constant $\eps_{k,m|n}(\CS{k}{n})$
is non-zero `most of the time'. Recall that this quantity are typically functions of the whole trajectory $\chunk{y}{0}{n}$.
It is not always easy to find such sets because the definition of the coupling constant involves the product $\FK{k}{n}$ forward smoothing kernels, which
is not easy to handle. In some situations (but not always), it is
possible to identify appropriate sets from the properties of the unconditional transition kernel $Q$.
\begin{defi}[Strong small set]
A set $C \in \Xsigma$ is a \emph{strong small set} for the transition kernel $Q$, if there exists a measure $\nu_C$ and constants
$\sigma_-(C) > 0$ and $\sigma_+(C) < \infty$ such that, for all $x \in C$ and $A \in \Xsigma$,
\begin{equation}
\label{eq:strong-small-set}
\sigma_-(C) \nu_C(A) \leq Q(x,A) \leq \sigma_+(C) \nu_C(A) \eqsp.
\end{equation}
\end{defi}
The following Lemma helps to characterize appropriate sets where coupling may occur with a positive probability from products of strong small sets.
\begin{prop}
\label{prop:strong-small-set}
Assume that $C$ is a strong small set.  Then, for any $n$ and any $k \in \{0, \dots, n\}$,
$C \times C$ is a coupling set for the forward smoothing kernels $\FK{k}{n}$; more precisely,
there exists a probability distribution $\nu_{k|n}$ such that, for any $A \in \Xsigma$,
\[
\inf_{x \in C} \FK{k}{n}(x,A) \geq \frac{\sigma_-(C)}{\sigma_+(C)} \nu_{k|n}(A)
\]
\end{prop}
\begin{proof} The proof is postponed to the appendix. \end{proof}
Assume that $\Xset = \rset^d$, and that the kernel satisfies the \emph{pseudo-mixing} condition \eqref{eq:pseudo-mixing-kernel}.
Let $C$ be a compact set $C$ with diameter $d = \diam(C)$ large enough so that \eqref{eq:pseudo-mixing-kernel} is satisfied. Then,
for any $n$ and any $k \in \{0, \dots, n\}$, $\bar{C}= C \times C$ is a coupling set for $\FK{k}{n}$, and $\eps(\bar{C})$ may be chosen
to be equal to $\eps_-(d)/\eps_+(d)$.
\cite{legland:oudjane:2003} gives non-trivial examples of pseudo-mixing Markov chains which are not uniformly ergodic.
Nevertheless, though the existence of small sets is automatically guaranteed for phi-irreducible Markov chains,
the conditions imposed for the existence of a strong small set are much more stringent.
As shown below, it is sometimes worthwhile to consider coupling set which are much larger than products of strong small sets.

\section{Coupling over the whole state-space}
\label{sec:uniform-accessibility}
The easiest situation is when the coupling constant of the whole state space $\eps_{k,m|n}(\Xset \times \Xset)$
is away from zero for sufficiently many trajectories $\chunk{y}{0}{n}$; for unconditional Markov chains, this property occurs
when the chain in uniformly ergodic (\ie\ satisfies the Doeblin condition). This is still the case here, through now the constants may depend
on the observations $Y$'s.  As stressed in the discussion, perhaps surprisingly, we will find non trivial examples where
the coupling constant $\epsilon_{k,m|n}(\Xset \times \Xset)$ is bounded away from zero for all $\chunk{y}{0}{n}$,
whereas the underlying unconditional Markov chain is \emph{not} uniformly geometrically ergodic.
We state without proof the following elementary result.
\begin{thm}
\label{thm:uniform-ergodicity-posterior}
Let $n$ be an integer and $m \geq 1$.  Then,
\[
\tvnorm{\filt[\Xinit]{n} - \filt[\Xinit']{n}}  \leq \prod_{k=0}^{\lfloor n/\ell \rfloor} \left\{1 - \eps_{k,m|n}(\Xset \times \Xset) \right\} \eqsp.
\]
\end{thm}
\begin{remark}
Consider the case where the kernel is uniformly ergodic,
\ie\
$$  \sigma_- \eqdef \inf_{(x,x') \in \Xset \times \Xset}
q(x,x')  > 0 \quad \text{and} \quad \sigma_+ \eqdef \sup_{(x,x') \in \Xset \times \Xset}
q(x,x') <\infty \eqsp. $$
One may thus take $m=1$ and, using Proposition \ref{prop:strong-small-set} $\epsilon_{k,1|n}(\Xset \times \Xset) \eqdef  \sigma_-/\sigma_+$.
In such a case, $\tvnorm{\filt[\Xinit]{n} - \filt[\Xinit']{n}}  \leq (1- \sigma_-/\sigma_+)^n$.
\end{remark}

To go beyond this example, we have to find verifiable conditions upon which we may ascertain that $\Xset \times \Xset$ is an $m$-coupling set.
\begin{defi}[Uniform accessibility]
Let $k,\ell,n$ be integers satisfying $\ell \geq 1$ and $k \in \{0, \dots, n-\ell\}$. A set $C$
is uniformly accessible for the forward smoothing kernels $\FK{k,\ell}{n}$ if there exists a constant $\kappa_{k,\ell}(C) > 0$ satisfying,
\begin{equation}
\label{eq:uniform-accessibility-condition}
\inf_{x\in \Xset} \FK{k,\ell}{n}(x,C) \geq \kappa_{k,\ell}(C) \eqsp.
\end{equation}
\end{defi}
The next step is to find conditions upon which a set is uniformly accessible.
For any set $A \in \mathcal{B}(\Xset)$,
define the function $\alpha: \Yset^{\ell} \to [0,1]$
\begin{equation}
\label{eq:definition-alpha}
\alpha(\chunk{y}{1}{\ell};A)
\eqdef \inf_{x_0,x_{\ell+1} \in \Xset \times \Xset} \frac{\W{\chunk{y}{1}{\ell}}{x_0,x_\ell}{A}}{\W{\chunk{y}{1}{\ell}}{x_0,x_{\ell+1}}{\Xset}}
= \left(1 + \tilde{\alpha}(\chunk{y}{1}{\ell};A)   \right)^{-1} \eqsp,
\end{equation}
where we have set
\begin{multline}
\label{eq:definition-W}
\W{\chunk{y}{1}{\ell}}{x_0,x_{\ell+1}}{A} \eqdef \\
\idotsint  \prod_{i=0}^{\ell} q(x_{i-1},x_{i}) g(x_i,y_i) q(x_{\ell},x_{\ell+1}) \1_A(x_{\ell}) \mu(\rmd \chunk{x}{1}{\ell}) \eqsp.
\end{multline}
and
\begin{equation}
\label{eq:tilde-alpha}
\tilde{\alpha}(\chunk{y}{1}{\ell-1};A) \eqdef \sup_{x_0,x_\ell \in \Xset \times \Xset} \frac{\W{\chunk{y}{1}{\ell-1}}{x_0,x_\ell}{A^c}}{\W{\chunk{y}{1}{\ell-1}}{x_0,x_\ell}{A}} \eqsp.
\end{equation}
Of course, the situations of interest are when $\alpha(\chunk{y}{1}{\ell-1};A)$ is positive or, equivalently,
$\tilde{\alpha}(\chunk{y}{1}{\ell-1};A) < \infty$. In such case, we may prove the following uniform accessibility condition:
\begin{prop}
For any integer $n$ and any $k \in \{0, \dots, n-\ell\}$,
\label{prop:uniform-accessibility}
\begin{equation}
\label{eq:minorationunifQ}
\inf_{x \in \Xset} \FK{k,\ell}{n}(x,C) \geq   \alpha(\chunk{Y}{k+1}{k+\ell};C) \eqsp.
\end{equation}
If in addition $C$ is a strong small set for $Q$, then $\Xset \times \Xset$ is a $(\ell+1)$-coupling set,
\begin{equation}
\label{eq:coupling-set}
\inf_{x \in \Xset} \FK{k+\ell+1}{n} \dots \FK{k}{n}(x,A) \geq \frac{\sigma_-(C)}{\sigma_+(C)} \alpha(\chunk{Y}{k+1}{k+\ell};C) \eqsp.
\end{equation}
\end{prop}
The proof is given in Section \ref{sec:proof-strong-small}.


\subsection{Examples}

\subsubsection{Bounded noise}
Assume that a Markov chain $\Xproc$ in $\Xset= \rset^{d_X}$ is observed in a bounded noise.
The case of bounded error is of course particular, because the observations of
the $Y$'s allow to locate the corresponding $X$'s within a set.
More precisely, we assume that $\Xproc$ is a Markov chain with transition kernel $Q$ having density $q$
\wrt\ the Lebesgue measure and $Y_k = b(X_k) + V_k$ where,
\begin{itemize}
\item $\{V_k\}$ is an i.i.d., independent of $\{X_k\}$, with density $p_V$. In addition, $p_V(|x|)= 0$ for $|x| \geq M$.
\item
\label{ass:boundederror2} the transition density $(x,x') \mapsto q(x,x')$ is strictly positive and continuous.
\item
\label{ass:levelset} The level sets of $b$, $ \{ x \in \Xset : |b(x)| \leq K \}$ are compact.
\end{itemize}
This case has already been considered by  \cite{budhiraja:ocone:1997}, using projective Hilbert metrics
techniques. We will compute an explicit lower bound for the coupling constant $\epsilon_{k,2|n}(\Xset \times \Xset)$, and will then prove,
under mild additional assumptions on the distribution of the $Y$'s that the chain forgets its initial conditions geometrically fast.
For $y \in \Yset$, denote $C(y) \eqdef \{ x \in \Xset, |b(x)| \leq |y|+M \}$. Note that, for any $x \in \Xset$ and
$A \in \Xsigma$,
\begin{multline*}
\FK{k+1}{n} \FK{k}{n}(x,A) = \\
\frac{\iint q(x,x_{k+1}) g(x_{k+1},Y_{k+1}) q(x_{k+1},x_{k+2}) g(x_{k+2},Y_{k+2}) \1_A(x_{k+2}) \backvar{k+2}{n}(x_{k+2}) \rmd x_{k+2}}{\iint q(x,x_{k+1}) g(x_{k+1},Y_{k+1}) q(x_{k+1},x_{k+2}) g(x_{k+2},Y_{k+2})  \backvar{k+2}{n}(x_{k+2}) \rmd x_{k+2}}
\end{multline*}
Since $q$ is continuous and positive, for any compact sets $C$ and $C'$, $\inf_{C\times C'} q(x,x') > 0$ and
$\sup_{C \times C'} q(x,x') < \infty$. On the other hand, because the observation noise is bounded, $g(x,y)= g(x,y) \1_{C(y)}(x)$.
Therefore,
\[
\FK{k+1}{n} \FK{k}{n}(x,A) \geq \rho(Y_{k+1},Y_{k+2}) \nu_{k|n}(A) \eqsp,
\]
where
\[
\rho(y,y')= \frac{\inf_{C(y) \times C(y')} q(x,x')}{\sup_{C(y) \times C(y')} q(x,x')} \eqsp,
\]
and
\[
\nu_{k|n}(A) \eqdef \frac{\int g(x_{k+2},Y_{k+2}) \1_A(x_{k+2}) \backvar{k+2}{n}(x_{k+2}) \nu(\rmd x_{k+2})}{\int g(x_{k+2},Y_{k+2})  \backvar{k+2}{n}(x_{k+2}) \nu(\rmd x_{k+2})} \eqsp.
\]
By applying Theorem \ref{thm:uniform-ergodicity-posterior}, we obtain that
\[
\tvnorm{\filt[\Xinit]{n}- \filt[\Xinit']{n}} \leq \prod_{k=0}^{\lfloor n/2 \rfloor} \{1 - \rho(Y_{2k},Y_{2k+1})\} \eqsp.
\]
Hence, the Markov chain is geometrically ergodic if
\[
\liminf_{n \to \infty} n^{-1} \sum_{k=0}^{\lfloor n/2 \rfloor} \rho(Y_{2k},Y_{2k+1})  > 0 \eqsp, \quad a.s. \eqsp.
\]
This property holds under many different assumptions on the observations $\chunk{Y}{0}{n}$ and in particular, if the observations follow a model which `approximately equal' to the assumed one.
\subsubsection{Functional autoregressive in noise}
\label{sec:FAR}
It is also of interest to consider cases
where both the $X$'s and the $Y$'s are unbounded. We consider a non-linear non-Gaussian
state space model (borrowed from \cite[Example 5.8]{legland:oudjane:2003}).
We assume that $X_0 \sim \Xinit$ and for $k \geq 1$,
\begin{align*}
X_k &= a(X_{k-1}) + U_k \eqsp, \\
Y_k &= b(X_k) + V_k \eqsp,
\end{align*}
where $\{U_k\}$ and $\{V_k\}$ are two independent sequences of  random variables,
with probability densities $\bar p_U$ and $\bar p_V$ \wrt\  the
Lebesgue measure on $\Xset = \rset^{d_X}$ and $\Yset = \rset^{d_Y}$, respectively. In addition, we assume that
\begin{itemize}
\item For any $x \in \Xset = \rset^{d_X}$,  $\bar p_U(x)=p_U(|x|)$ where $p_U$ is a bounded, bounded away from zero on $[0,M]$,
is non increasing on $[M,\infty[$, and for some positive constant $\gamma$,
\begin{equation}
\label{eq:condition-noise}
\frac{p_U(\alpha+\beta)}{p_U(\alpha) p_U(\beta)} \geq \gamma > 0  \eqsp.
\end{equation},
\item the function $a$ is  Lipshitz, \ie\ there exists a positive constant $a_+$ such that
$| a(x) - a(x')| \leq a_+ |x-x'|$, for any $x,x' \in \Xset$,
\item  the function $b$ is one-to-one differentiable and its Jacobian is bounded and bounded away from zero.
\item For any $y \in \Yset = \rset^{d_Y}$,  $\bar p_V(y)=p_V(|y|)$ where $p_V$ is a bounded positive lower semi-continuous function,
$p_V$ is non increasing on $[M,\infty[$, and satisfies
\begin{equation}
\label{eq:condition-tail}
\Upsilon \eqdef \int_0^\infty [p_U(x)]^{-1} p_V(b_- x) [p_U(a_+ x)]^{-1} \rmd x < \infty \eqsp,
\end{equation}
where $b_-$ is the lower bound for the Jacobian of the function $b$.
\end{itemize}
The condition on the state noise $\{U_k\}$
is satisfied by Pareto-type, exponential and logistic densities but obviously not by Gaussian density, because
the tails are in such case too light.

The fact that the tails of the state noise $U$  are heavier than the tails of the observation noise $V$
(see \eqref{eq:condition-tail}) plays a key role in the derivations that follow.
In Section \ref{sec:ProxSet} we consider a case where this restriction is not needed (e.g., normal).
\marg{it is even worse than this here. This is because the condition on the state space noise is not satisfied
by Gaussian noise}

The following technical lemma (whose proof is postponed to section
\ref{sec:proof:FAR}), shows that any set with finite diameter is a strong small set.
\begin{lemma}
\label{lem:FAR}
Assume that $\diam(C) < \infty$. Then, for all $x_0 \in C$ and $x_1 \in \Xset$,
\begin{equation}
\label{eq:strong-small-set-FAR}
\epsilon(C) h_C(x_1) \leq q(x_0,x_1) \leq \epsilon^{-1}(C) h_C(x_1) \eqsp,
\end{equation}
with
\begin{align}
\label{eq:definition-epsilon}
\epsilon(C) &\eqdef \gamma p_U(\diam(C)) \wedge \inf_{u \leq \diam(C)+M} p_U(u) \wedge \left(\sup_{u \leq \diam(C)+M} p_U(u) \right)^{-1} \eqsp,  \\
\label{eq:definition-h-FAR}
h_C(x_1) &\eqdef \1 ( d(x_1, a(C) ) \leq M) +  \1 ( d(x_1, a(C) ) > M) p_U(|x_1 - a(z_0)|) \eqsp,
\end{align}
where $\gamma$ is defined in \eqref{eq:condition-noise} and $z_0$ is an arbitrary element of $C$.   In addition, for all $x_0 \in \Xset$ and $x_1 \in C$,
\begin{equation}
\label{eq:decomposability-condition}
\nu(C) k_C(x_0) \leq  q(x_0,x_1) \eqsp,
\end{equation}
with
\begin{align}
\label{eq:definition-nu}
\nu(C) &\eqdef \inf_{|u| \leq \diam(C)+M} p_U \eqsp, \\
\label{eq:definition-k-FAR}
k_C(x_0) &\eqdef \1 (d (a(x_0),C) < M) + \1 (d (a(x_0),C) \geq M) p_U(|z_1-a(x_0)|) \eqsp,
\end{align}
where $z_1$ is an arbitrary point in $C$.
\end{lemma}


By Lemma \ref{lem:FAR}, the denominator of \eqref{eq:tilde-alpha} is lower bounded by
\begin{equation}
\label{eq:lowerbound-denominator}
W[y](x_0,x_2;C) \geq \epsilon( C) \; \nu(C) k_C(x_0) h_C(x_2) \int_C  g(x_1,y)  \rmd x_1 \eqsp.
\end{equation}
Therefore, we may bound  $\tilde{\alpha}(y_1,C)$, defined in \eqref{eq:tilde-alpha}, by
\begin{multline}
\label{eq:cond4}
\tilde{\alpha}(y_1,C) \leq \left( \epsilon(C) \; \nu(C) \; \int_C g(x_1,y_1) \rmd x_1 \right)^{-1} \; \\ \times
\sup_{x_0,x_2 \in \Xset} [k_C(x_0)]^{-1} [h_C(x_2)]^{-1} W[y_1](x_0,x_2;C^c)   \eqsp.
\end{multline}
In the sequel, we choose $C = C_K(y) \eqdef \{ x, |x - b^{-1}(y) | \leq K \}$, where $K$ is a constant which will be chosen later. Since, by construction, the diameter of the set
$C_K(y)$ is $2K$ uniformly \wrt\ $y$, the constants $\epsilon(C_K(y))$  (defined in \eqref{eq:definition-epsilon}) and
$\nu(C_K(y))$ (defined in \eqref{eq:definition-nu}) are functions of $K$ only and are therefore uniformly bounded from below \wrt\ $y$.
We will first show that, for $K$ large enough, $\int_{C_K(y)}  g(x_1,y) \rmd x_1$ is uniformly bounded from below, as shown in the
following Lemma (whose proved is postponed to Section \ref{sec:proof:FAR}). The following two Lemmas bound the terms appearing in the RHS of \eqref{eq:cond4}.
\begin{lemma}
\label{lem:borne-inf-uniforme}
\[
\lim_{K \to \infty} \inf_{\Yset} \int_{C_K(y)} p_V(|y-b(x)|) \rmd x > 0 \eqsp.
\]
\end{lemma}
We set $z_0= b^{-1}(y)$ in the definition \eqref{eq:definition-h-FAR} of $h_{C(y)}$ and $z_1= b^{-1}(y)$ in the definition \eqref{eq:definition-k-FAR}.
We denote
\begin{multline}
\label{eq:definition-I}
I_K(x_0,x_2;y) \eqdef [k_{C_K(y)}(x_0)]^{-1} [h_{C_K(y)}(x_2)]^{-1} \\
\times \int_{C_K^c(y)} p_U(|x_1-a(x_0)|) p_V(|y-b(x_1)|) p_U(|x_2 - a(x_1)|) \rmd x_1 \eqsp.
\end{multline}
The following Lemma shows that  $K$ may be chosen large enough so that $I_K(x_0,x_2,y)$ is uniformly bounded over $x_0,x_2$ and $y$.
\begin{lemma}
\label{lem:util:FAR}
\begin{equation}
\label{eq:key-property-FAR}
\limsup_{K \to \infty} \sup_{y \in \Yset} \sup_{(x_0,x_2) \in \Xset \times \Xset} I_K(x_0,x_2;y) < \infty \eqsp.
\end{equation}
\end{lemma}
The proof is postponed to Section \ref{sec:proof:FAR}.

\section{Pairwise drift conditions}
\label{sec:ProxSet}

\subsection{The pair-wise drift condition}
In the situations where coupling over the whole state-space leads to trivial result, one may still
use the coupling argument, but this time over smaller sets. In such cases, however, we need a device to control the
return time of the joint chain to the set where the two chains are allowed to couple.
In this section we obtain results that are general enough to
include the autoregression model with Gaussian innovations and Gaussian measurement error.
Drift conditions are used to obtain bounds on the coupling time.
Consider the following drift condition.
\begin{defi}[Pair-wise drift conditions toward a set]
\label{def:pair-wise-drift-condition}
Let $n$ be an integer and $k \in \{0,\dots, n-1\}$ and let $\CS{k}{n}$ be a set valued function $\CS{k}{n}: \Yset^{n+1} \to \Xsigma \times \Xsigma$.
We say that the forward smoothing kernel $\FK{k}{n}$ satisfies  the \emph{pair-wise drift condition}
toward the set $\CS{k}{n}$ if  there exist functions $V_{k|n}: \Xset\times\Xset \times \Yset^{n+1} \to \rset$, $V_{k|n}\geq
1$, functions $\lambda_{k|n}: \Yset^{n+1} \to [0,1)$, $\rho_{i|n}: \Yset^{n+1} \to \rset^+$
such that, for any sequence $\chunk{y}{0}{n} \in \Yset^n$,
\begin{align}
\label{eq:pair-wise-drift-1}
& \RFK{k}{n} V_{k+1|n}(x,x') \leq  \rho_{k|n} && (x,x') \in \CS{k}{n} \\
\label{eq:pair-wise-drift-2}
& \bFK{k}{n}  V_{k+1|n}(x,x') \leq \lambda_{k|n} V_{k|n}(x,x') && (x,x') \not \in \CS{k}{n} \eqsp.
\end{align}
where $\RFK{k}{n}$ is defined in \eqref{eq:definition-residual} and $\bFK{k}{n}$ is defined in \eqref{eq:definition-product-kernel}.
\end{defi}
We set  $\eps_{k|n}= \eps_{k|n}(\CS{k}{n})$, the coupling constant of the set $\CS{k}{n}$, and we denote
\begin{equation}
\label{eq:definition-B}
B_{k|n} \eqdef 1 \vee \rho_{k|n} (1 - \eps_{k|n}) \lambda_{k|n} \eqsp.
\end{equation}
For any vector $\{ a_{i,n}\}_{1 \leq  i \leq n}$, denotes by $[\downarrow a]_{(i,n)}$ the  $i$-th largest order statistics,
\ie\ $[\downarrow a]_{(1,n)} \geq [\downarrow a]_{(2,n)} \geq \dots \geq [\downarrow a]_{(n,n)}$ and $[\uparrow a]_{(i,n)}$ the $i$-th smallest order statistics,
\ie\ $[\uparrow a]_{(1,n)} \leq [\uparrow a]_{(2,n)} \leq \dots \leq [\uparrow a]_{(n,n)}$.
\begin{thm}
\label{theo:coupling-drift}
Let $n$ be an integer. Assume that for each $k \in \{0, \dots, n-1\}$, there exists a set-valued function $\CS{k}{n}:
\Yset^{n+1} \to \Xsigma \otimes \Xsigma$ such that the forward smoothing kernel $\FK{k}{n}$ satisfies the pairwise drift condition
toward the set $\CS{k}{n}$. Then,
for any probability $\Xinit,\Xinit'$ on $(\Xset,\mathcal{B}(\Xset))$,
\begin{equation}
\label{eq:coupling-drift-1}
\tvnorm{\filt[\Xinit]{n} - \filt[\Xinit']{n}}  \leq  \min_{1 \leq m \leq n} A_{m,n}
\end{equation}
where
\begin{equation}
\label{eq:definition-Amn}
A_{m,n} \eqdef  \prod_{i=1}^m (1-[\uparrow \epsilon]_{(i|n)}) + \\
\prod_{i=0}^{n} \lambda_{i|n} \prod_{i=0}^m [\downarrow B]_{(i|n)} \Xinit \otimes \Xinit' (V_0)
\end{equation}
\end{thm}
The proof is in section \ref{sec:proof:theo:coupling-drift}.

\begin{cor}
If there exists a sequence $\{ m(n) \}$ of integers satisfying,  $m(n) \leq n$ for any integer $n$,
$\lim_{n \to \infty} m(n) = \infty$, and, $\PP^Y$-\as\,
\[
\limsup \left( \sum_{i=0}^{m(n)} \log (1-[\uparrow \epsilon]_{(i|n)}) +  \sum_{i=0}^n \log \lambda_{i|n} +  \sum_{i=0}^{m(n)} \log [\downarrow B_{(i,n)}] \right) = - \infty \eqsp,
\]
then
\[
\limsup_n \tvnorm{\filt[\Xinit]{n} - \filt[\Xinit']{n}}  \cas 0 \eqsp, \quad \PP^Y-\as\ \eqsp.
\]
\end{cor}

\begin{cor}
If there exists a sequence $\{ m(n) \}$ of integers such that  $m(n) \leq n$ for any integer $n$,
$\liminf m(n)/n = \alpha > 0$ and $\PP^Y$-\as\,
\[
\limsup \left( \frac1n\sum_{i=0}^{m(n)} \log (1-[\uparrow \epsilon]_{(i|n)}) + \frac 1n \sum_{i=1}^n \log \lambda_{i|n} + \frac 1n \sum_{i=1}^{n-m(n)} \log [\downarrow B_{(i|n)}] \right) \leq -\lm \eqsp,
\]
then  there exists $\nu \in (0,1)$ such that
\[
\nu^{-n} \tvnorm{\filt[\Xinit]{n}-\filt[\Xinit']{n}} \cas 0 \eqsp, \quad \PP^Y-\as\ \eqsp.
\]
\end{cor}

\subsection{Examples}

\subsubsection{Gaussian autoregression}
\label{sec:GaussianAR}
Let
\begin{align*}
    X_i &= \al X_{i-1} + \sigma U_i\\
    Y_i &= X_i+ \tau V_i
\end{align*}
where $|\al|<1$ and $\{U_i\}_{i\geq 0}$ and $\{V_i\}$ are \iid\ standard Gaussian
and are independent from $X_0$. Let $n$ be an integer and $k \in \{0, \dots, n-1\}$.
The backward functions are given by
\begin{equation}
\label{eq:backward-functions}
\backvar{k}{n}(x) \propto \exp\left( - (\alpha x-m_{k|n})^2/(2 \rho^2_{k|n}) \right) \eqsp,
\end{equation}
where $m_{k|n}$ and $\rho_{k|n}$ can be computed for $k= \{0, \dots, n-2 \}$ using the following backward recursions (see \eqref{eq:rec:beta})
\begin{align}
\label{eq:definition-m-rho}
& m_{k|n}= \frac{\rho^2_{k+1|n} Y_{k+1} + \alpha \tau^2 m_{k+1|n}}{\rho_{k+1|n}^2 + \alpha^2 \tau^2} \eqsp,
& \rho^2_{k|n} = \frac{(\tau^2 + \sigma^2) \rho_{k+1|n}^2  + \alpha^2 \sigma^2 \tau^2}{\rho_{k+1|n}^2 + \alpha^2 \tau^2} \eqsp.
\end{align}
initialized with $m_{n-1|n}= Y_n$ and $\rho_{n-1|n}= \sigma^2+\tau^2$.
The conditional transition kernel $\FK{i}{n}(x,\cdot)$ has a density \wrt\ to the Lebesgue measure given by
$\phi(\cdot;\mu_{i|n}(x),\gamma^2_{i|n})$, where $\phi(z;\mu,\sigma^2)$ is the density of a Gaussian random variable with mean $\mu$ and variance $\sigma^2$ and
\begin{align*}
&\mu_{i|n}(x)= \frac{\tau^2 \rho_{i+1|n}^2 \alpha x + \sigma^2 \rho_{i+1|n}^2 Y_{i+1} + \sigma^2 \alpha \tau m_{i+1|n}}{(\sigma^2 + \tau^2) \rho_{i+1|n}^2 + \tau^2 \alpha^2 \sigma^2} \eqsp, \\
&\gamma^2_{i|n}= \frac{\sigma^2 \tau^2 \rho_{i+1|n}^2}{(\tau^2 + \sigma^2) \rho^2_{i+1|n} + \alpha^2 \tau^2 \sigma^2} \eqsp.
\end{align*}
From \eqref{eq:definition-m-rho}, it follows  that for any $i \in \{0, \dots, n-1 \}$,
$\sigma^2 \leq \rho_{i|n}^2 \leq \sigma^2+\tau^2$. This implies that, for any $(x,x') \in \Xset \times \Xset$, and any $i \in \{0, \dots, n-1 \}$,
the function $\mu_{i|n}$ is Lipshitz and with Lipshitz constant which is uniformly bounded by some $\beta < |\alpha|$,
\begin{equation}
\label{eq:definition-beta}
|\mu_{i|n}(x) - \mu_{i|n}(x')| \leq \beta  |x-x'| \eqsp, \quad \beta \eqdef |\alpha| \frac{\tau^2 (\sigma^2+\tau^2)}{(\sigma^2+\tau^2)^2+\tau^2 \alpha^2 \sigma^2} \eqsp,
\end{equation}
and that the variance is uniformly bounded
\begin{equation}
\label{eq:definition-gamma}
\gamma^2_- \eqdef \frac{\sigma^2 \tau^2}{(1+\alpha^2)\tau^2+\sigma^2} \leq \gamma^2_{i|n} \leq \gamma^2_+ \eqdef \frac{\sigma^2 \tau^2 (\sigma^2+\tau^2)}{(\tau^2 + \sigma^2)^2 + \alpha^2 \tau^2 \sigma^2} \eqsp.
\end{equation}
Therefore, for any $c<\en$, all sets of the form
\begin{equation}
\label{eq:defintion-coupling-set}
C \eqdef \left\{(x,x') \in \Xset \times \Xset :\; |x-x'| \leq c \right\} \eqsp,
\end{equation}
are coupling sets. Note indeed that, for any $i \in \{0, \dots, n-1\}$,
\[
\frac{1}{2} \tvnorm{\FK{i}{n}(x,\cdot) - \FK{i}{n}(x',\cdot)} =
2 \mathrm{erf} \left( \gamma_{i|n}^{-1} |\mu_{i|n}(x)-\mu_{i|n}(x')| \right) \leq 2 \mathrm{erf}(\gamma_-^{-1} \beta c) \eqsp,
\]
where $\mathrm{erf}$ is the error function.
More precisely, for any $(x,x') \in C$ and any integer $n$ and any $i \in \{0,\dots,n-1\}$,
\begin{equation}
\label{eq:definition-epsilon-1}
\FK{i}{n}(x,A) \wedge \FK{i}{n}(x',A) \geq \eps \nu_{i,1|n}(x,x';A) \eqsp, \text{where} \; \eps \eqdef \left(1-2\mathrm{erf}(\gamma_-^{-1} \beta c)\right) \eqsp,
\end{equation}
and $\nu_{i,1,n}$ is defined as in \eqref{eq:minorizing-probability}. For $c$ large enough, the drift condition is satisfied with
$V(x,x')=1+(x-x')^2$:
\[
\bFK{i}{n} V(x,x') =  1+ \left\{\mu_{i|n}(x) - \mu_{i|n}(x')\right\}^2 + 2 \gamma_{i|n}^2 \leq 1 + \beta^2 |x-x'|^2 + \gamma_+^2 \eqsp.
\]
The condition \eqref{eq:pair-wise-drift-1} with
\begin{equation}
\label{eq:definition-rho}
\rho_{i|n} \leq  \rho \eqdef (1-\epsilon)^{-1} \left(1 + \beta^2 c^2 + \gamma_+^2 \right) \eqsp,
\end{equation}
where $c$ is the width of the coupling set in \eqref{eq:defintion-coupling-set}.
The condition \eqref{eq:pair-wise-drift-2} is satisfied with $\lambda_{i|n}= \tilde{\beta}^2$
for any $\tilde{\beta}$ and $c$ satisfying $\beta < \tilde{\beta} < 1$  and $c^2>(1-\tilde{\beta}^2+\gamma_+^2)/(\tilde{\beta}^2- \beta^2)$.
it is worthwhile to note that all these bounds are uniform \wrt\ $n$, $i \in \{0, \dots, n-1 \}$ and realization of the
observations $\chunk{y}{0}{n}$. Therefore, for any $m \in \{0,\dots, n\}$, we may take upper bound $A_{m,n}$ (defined in \eqref{eq:definition-Amn}) by
\[
A_{m,n} \leq (1-\eps)^m + B^m \tilde{\beta}^{2n} (1 + 2 \int \Xinit(dx) x^2 + 2 \int \Xinit'(dx) x^2)
\]
with $B= 1 \vee \rho (1-\eps)\tilde{\beta}^2$, where $\eps$ is defined in \eqref{eq:definition-epsilon-1}, $\rho$ is defined in \eqref{eq:definition-rho}.
Taking $m= [\delta n]$ for some $\delta > 0$ such that $B^\delta \tilde{\beta}^2 < 1$, this upper bound may be shown to go to zero exponentially
fast and uniformly \wrt\ the observations $\chunk{y}{0}{n}$.

\subsubsection{State space models with strongly modal distributions}
The Gaussian example can be generalized to the more general
case where the distribution of the state noise and the measurement noise
are strongly unimodal. Recall that a density is strongly modal if  the log of
its density is concave.

First note that if $f$ and $g$ are two strongly unimodal density,
then the density $h= fg/\int fg$ is
also strongly unimodal, with mode that lies between the two modes;
its second-order derivative of $\log h$ is  smaller that the sum
of the second-order derivative of $\log f$ and $\log g$. Let the state noise density be denoted
by $p_U(\cdot)=\rme^{\varphi(\cdot)}$ and that of the
measurements' errors be $p_V(\cdot)=\rme^{\psi(\cdot)}$.
Define by  the recursion operating on the decreasing indices
\begin{equation}
\label{eq:definition-bar-beta}
\bar{\beta}_{i|n}(x) = p_V(y_i-x) \int q(x,x_{i+1}) \bar{\beta}_{i+1|n}(x_{i+1}) \rmd x_{i+1} \eqsp,
\end{equation}
with the initial condition $\bar{\beta}_{n|n}(x)= p_V(y_n-x)$. These functions are the conditional
distribution of the observations $\chunk{Y}{i}{n}$ given $X_i=x$. They are related to the backward function
through the relation $\bar{\beta}_{i|n}(x)\eqdef \backvar{i}{n}(x) p_V(y_i-x)$. We denote $\psi_{i|n}(x) \eqdef \log \bar{\beta}_{i|n}(x)$.
Now,
\[
\psi_{i|n}(x) = \psi(Y_i-x) +\log \int p_U(z-\al x) \bar{\beta}_{i+1|n}(z) \rmd z \eqsp.
\]
Under the stated assumptions, the forward
smoothing kernel $\FK{i}{n}$ has a density with respect to the Lebesgue measure which is given by
\begin{multline}
\label{eq:forward-smoothing-density}
\fk{i}{n}(x_i, x_{i+1}) = \\
p_U(x_{i+1} -\alpha x_i) \bar{\beta}_{i+1|n}(x_{i+1})/ \int p_U(z - \alpha x_i) \bar{\beta}_{i+1|n}(z) \rmd z \eqsp.
\end{multline}
Denote by $\widetilde\PCov_{i|n,x}$  the covariance function \wrt\
the forward smoothing kernel density.
We recall that for any probability distribution $\PP$ on $(\Xset,\Xsigma)$ and any two increasing measurable functions $f$ and $g$
which are square integrable \wrt\ $\PP$, the covariance of $f$ and $g$ \wrt\ $\PP$,
is non-negative. Hence,
 \eqsplit[sum_induc]{
    \psi_{i|n}&''(x)
    \\
    &=\psi''(Y_i-x) + \al^2\frac{\int p_U''(z-\al x) \bar{\beta}_{i+1|n}(z)\,dz} {\int p_U(z-\al x) \bar{\beta}_{i+1|n}(z)\,dz }
     -\al^2\Bigl( \frac{\int p_U'(z-\al x) \bar{\beta}_{i+1|n}(z)\,dz}{\int p_U(z-\al x)
     \bar{\beta}_{i+1|n}(z)\,dz}\Bigr)^2
     \\
     &=\psi''(Y_i-x) - \al^2\frac{\int p_U'(z-\al x) \bar{\beta}_{i+1|n}'(z)\,dz} {\int p_U(z-\al x) \bar{\beta}_{i+1|n}(z)\,dz }
     \\
     &\hspace{3em}+\al^2\Bigl( \frac{\int p_U'(z-\al x) \bar{\beta}_{i+1|n}(z)\,dz}{\int p_U(z-\al x)
     \bar{\beta}_{i+1|n}(z)\,dz}\Bigr)\Bigl( \frac{\int p_U(z-\al x) \bar{\beta}'_{i+1|n}(z)\,dz}{\int p_U(z-\al x)
     \bar{\beta}_{i+1|n}(z)\,dz}\Bigr)
     \\
    &=\psi''(Y_i-x) - \al^2 \widetilde\PCov_{i|n,x} \bigl( \varphi'(\cdot-\al x), \psi'_{i+1|n}(\cdot)    \bigr)
    \\
    &\leq \psi''(Y_i-x),
  }
where we used a direct differentiation, integration by parts, and
the fact that both $\phi'$ and $\psi_{i+1|n}'$ are monotone
non-increasing functions (the last statement follows by applying
\eqref{sum_induc} inductively from $n$ backward).

We conclude that $\psi_{i|n}$ is strongly unimodal with curvature at
least as that of the original likelihood function.
Hence the curvature of the logarithm of the forward smoothing density is smaller than the sum of the curvature of the state and of the
measurement noise,
\begin{equation}
\label{eq:curvature-log-forward-smoothing-density}
\left[\log \fk{i}{n}(x_i,x_{i+1}) \right]'' \leq  \varphi''(x_{i+1} - \alpha x_i) + \psi''(Y_{i+1} - x_{i+1}) \leq - c  \eqsp,
\end{equation}
where
\begin{equation}
\label{eq:definition-c}
c = - \max_{x_{i+1}} \varphi''(x_{i+1}) +  \max_{x_{i+1}} \psi''(x_{i+1}) \eqsp.
\end{equation}
Lemma \ref{lem:varSU} shows that the variance of $X_{i+1}$ given $X_i$ and $\chunk{Y}{i+1}{n}$ is uniformly bounded
\[
v_{i|n}(x) \eqdef \int \left( x_{i+1} - \int x_{i+1} \fk{i}{n}(x,x_{i+1}) \rmd x_{i+1} \right)^2 \fk{i}{n}(x,x_{i+1}) \rmd x_{i+1}
\leq c^{-1} \eqsp.
\]
where $c$ is defined in \eqref{eq:definition-c}. Now let
\[
e_{i|n}(x) \eqdef \int x_{i+1} \fk{i}{n}(x,x_{i+1}) \rmd x_{i+1} \eqsp.
\]
Similarly as above
\[
    \frac{\rmd e_{i|n}}{\rmd x} (x) = -\alpha\,\widetilde\PCov_{i|n,x}\left(Z,\varphi'(Z-\alpha x) \right) \eqsp.
\]
Note that $x_{i+1} \mapsto e_{i|n}(x)-x_{i+1}$, $x_{i+1} \mapsto \varphi'(x_{i+1}-\alpha x)$, and $x_{i+1} \mapsto \psi'_{i+1|n}(x_{i+1})$ are
monotone non-increasing and therefore their correlation is positive \wrt\ any probability measure. Hence
\begin{align*}
&\left|\frac{\rmd e_{i|n}}{\rmd x} (x) \right|\\
&  =|\alpha|\frac{\int \left(e_{i|n}(x)-x_{i+1} \right) \varphi'(x_{i+1} - \alpha x) \rme^{\varphi(x_{i+1}-\alpha x)+\psi_{i+1|n}(x_{i+1})}\rmd x_{i+1}}
    {\int \rme^{\varphi(x_{i+1}-\al x)+\psi_{i+1|n}(x_{i+1})} \rmd x_{i+1}} \\
&  \leq |\alpha| \frac{\int \left(e_{i|n}(x)-x_{i+1}\right) \left(\varphi'(x_{i+1}-\alpha x)+\psi'_{i+1|n}(x_{i+1})\right)
    \rme^{\phi(x_{i+1}-\alpha x)+\psi_{i+1|n}(x_{i+1})} \rmd x_{i+1}} {\int \rme^{\varphi(x_{i+1}-\alpha x)+\psi_{i+1|n}(x_{i+1})} \rmd x_{i+1}}
    \\
&  = |\alpha| \eqsp.
\end{align*}
by integration by parts. Put as before $V(x,x')= 1 + (x-x')^2$.
It follows from the discussion above that
\begin{align*}
\bFK{i}{n}V(x,x')=  1 + (e_{i|n}(x) - e_{i|n}(x'))^2 + v_{i|n}(x) + v_{i|n}(x') \eqsp,
\end{align*}
where $v_{i|n}(x)$ and $v_{i|n}(x')$ are uniformly bounded \wrt\ $x$ and $x'$ and
$|e_{i|n}(x) - e_{i|n}(x')| \leq \alpha |x-x'|$.  The rest of
the argument is like that for the normal-normal case.

We conclude the argument by stating and proving a lemma which was used above.
\begin{lemma}
\label{lem:varSU}
Suppose that $Z$ is a random variable with probability density function  $f$ satisfying
$\sup_x(\partial^2/\partial x^2)\log f \leq -c$. Then, $Z$ is square integrable and $\PVar(Z) \leq c^{-1}$.
\end{lemma}
\begin{proof}
Suppose, w.l.o.g., that the maximum of $f$ is at 0.
Under the stated assumption, there exist constants $a \geq 0$ and $b$ such that $f(x) \leq a \rme^{-c(x-b)^2}$. This implies that
$Z$ is quare integrable.
 Denote  $z \mapsto \zeta(z)=\log f(z)+cz^2/2$ which by assumption is a concave function. Let $m$ be the mean of $Z$.
\begin{multline*}
\PE[(Z-m)^2]= \int (z-m) z \rme^{\xi(z)-cz^2/2} \rmd z = \\  c^{-1} \int (z-m) \left(cz-\xi'(z)\right) \rme^{\xi(z)-cz^2/2} \rmd z  +
 c^{-1} \int (z-m) \xi'(z) \rme^{\xi(z)-cz^2/2} \rmd z.
\end{multline*}
By construction, $z \mapsto \xi'(z)$ is a non-increasing function.  Since the inequality $\PCov(\varphi(Z),\psi(Z)) \geq 0$ holds
for any two non-decreasing function $\varphi$ and $\psi$  which have  finite second moment, the second term in the RHS of the previous equation is
negative. Since $ \left(cz-\xi'(z)\right) \rme^{\xi(z)-cz^2/2} = -f'(z)$, the proof follows by integration by part:
\[
\PVar(Z) \leq -c^{-1} \int (z-m) f'(z) \rmd z = c^{-1} \int f(z) \rmd z = c^{-1} \eqsp.
\]
\end{proof}

\section{Proofs}
\label{sec:proof-strong-small}
\begin{proof}[Proof of Proposition \ref{prop:strong-small-set}]
The proof is similar to the one done in \cite{douc:moulines:ryden:2004}.
For $x \in C$, the condition, \eqref{eq:strong-small-set} implies that
\[
\sigma_-(C) \nu_C(dx') \leq \frac{d Q(x,\cdot)}{d \nu_C}(dx') \leq \sigma_+(C) \nu_C(dx') \eqsp.
\]
Plugging the lower and upper bounds in the numerator and the denominator of
\eqref{eq:smoothing:forward_kernel} yields,
\[
\FK{k}{n}(x_k,A) \geq \frac{\sigma_-}{\sigma_+} \frac{\int_A
\frac{d Q(x_k,\cdot)}{d \nu_C}(dx_{k+1})  \backvar{k+1}{n}(x_{k+1}) \mu(\rmd x_{k+1})}{\int_\Xset  \frac{d Q(x_k,\cdot)}{d \nu_C}(dx_{k+1}) \backvar{k+1}{n}(x_{k+1})
\mu(\rmd x_{k+1})}
\]
The result is established with
\[
\nu_{k|n}(A) \eqdef \frac{\int_A \frac{d Q(x_k,\cdot)}{d \nu_C}(dx_{k+1}) \backvar{k+1}{n}(x_{k+1})
\mu(\rmd x_{k+1})}{\int_\Xset \frac{d Q(x_k,\cdot)}{d \nu_C}(dx_{k+1}) \backvar{k+1}{n}(x_{k+1}) \mu(\rmd x_{k+1})} \eqsp.
\]
\end{proof}

\begin{proof}[Proof of proposition \ref{prop:uniform-accessibility}]
For any $x_i \in \Xset$,
\begin{align*}
& \CPP{X_{i+\ell} \in C}{X_i=x_i, \chunk{Y}{1}{n}} \\
&= \frac{\idotsint \W{\chunk{Y}{i+1}{i+\ell}}{x_i,x_{i+\ell+1}}{C}  \backvar{i+\ell+1}{n}(x_{i+\ell+1}) \mu(\rmd \chunk{x}{i+1}{i+\ell+1})}
{\idotsint \W{\chunk{Y}{i+1}{i+\ell}}{x_i,x_{i+\ell+1}}{\Xset} \backvar{i+\ell+1}{n}(x_{i+\ell+1}) \mu(\rmd \chunk{x}{i+1}{i+\ell+1})} \eqsp,
\\
&= \frac{\idotsint \frac{\W{\chunk{Y}{i+1}{i+\ell}}{x_i,x_{i+\ell+1}}{C}}{\W{\chunk{Y}{i+1}{i+\ell}}{x_i,x_{i+\ell+1}}{\Xset}} \W{\chunk{Y}{i+1}{i+\ell}}{x_i,x_{i+\ell+1}}{\Xset}
\backvar{i+\ell+1}{n}(x_{i+\ell+1})
\mu(\rmd x_{i+\ell+1})} {\idotsint \W{\chunk{Y}{i+1}{i+\ell}}{x_i,x_{i+\ell+1}}{\Xset} \backvar{i+\ell+1}{n}(x_{i+\ell+1}) \mu(\rmd x_{i+\ell+1})},
\end{align*}
where  $W$ is defined in \eqref{eq:definition-W}. The proof is concluded by noting that,
under the stated assumptions,
\[
\sup_{(x_i,x_{i+\ell+1}) \in \Xset \times \Xset} \frac{\W{\chunk{Y}{i+1}{i+\ell}}{x_i,x_{i+\ell+1}}{C}}{\W{\chunk{Y}{i+1}{i+\ell}}{x_i,x_{i+\ell+1}}{\Xset}} \geq \alpha(\chunk{Y}{i+1}{i+\ell};C) \eqsp,
\]
\end{proof}

\subsection{Proof of Theorem \ref{theo:coupling-drift}}
\label{sec:proof:theo:coupling-drift}
\begin{proof}
For notational simplicity, we drop the dependence  in the sample size $n$.
Denote $ N_{n} \eqdef \sum_{j=0}^n \1_{\bar{C}_j} (X_j,X'_j)$ and $\eps_i \eqdef \eps(\bar{C}_i)$.
For any $m \in \{1, \dotsc, n+1 \}$, we have:
\begin{multline}
\label{eq:key-coupling}
\PP^Y_{\Xinit,\Xinit',0} \left(T \geq n\right) \\
\quad \leq \PP^Y_{\Xinit,\Xinit',0} \left(T \geq n, N_{n-1} \geq m  \right) + \PP^Y_{\Xinit,\Xinit',0} \left(T \geq n, N_{n-1} < m \right) \eqsp.
\end{multline}
The first term on the RHS of the previous equation is the probability that we fail to couple the chains after at least $m$ independent trial.
it is bounded by
\begin{equation}
\label{eq:first-term}
\PP^Y_{\Xinit,\Xinit',0} \left(T \geq n, N_{n-1} \geq m  \right) \leq \prod_{i=1}^m \left(1-[\uparrow \eps]_{(i)} \right).
\end{equation}
where $[\uparrow \eps]_{(i)}$ are the smallest-order statistics of $(\eps_1,\dots,\eps_n)$.
We consider now the second term in the RHS of \eqref{eq:key-coupling}.
Set $B_j \eqdef 1 \vee \rho_j (1-\eps_j) \lambda^{-1}_j$.
On the event $\{ N_{n-1}\leq m-1 \}$,
\[
\prod_{j=1}^{n} B_j^{\1_{\bar{C}_j}(X_j,X'_j)} \leq \prod_{j=1}^{m-1} [\downarrow B]_{(j)} \eqsp,
\]
where $[\downarrow B]_{(j)}$ is the $j$-th largest order statistics of $B_1, \dots, B_n$.
Hence,
\[
\1 \{ N_{n-1}\leq m-1 \}  \leq \left(\prod_{j=1}^{n} B_j^{\1_{\bar{C}_j}(X_j,X'_j)} \right)^{-1} \, \prod_{j=1}^{m-1} [\downarrow B]_{(j)} \eqsp,
\]
which implies that:
\begin{equation}
\label{eq:1}
\PP^Y_{\Xinit,\Xinit',0} \left(T \geq n, N_{n-1} < m \right) \leq \prod_{j=1}^n \lambda_j \prod_{j=1}^{m_1} [\downarrow B]_{(j)}
\PE^Y_{\xi \otimes \xi' \otimes \delta_0}[M_n]
\end{equation}
where, for $k \in \{0, \dots,n \}$:
\begin{equation}
\label{eq:definitionMs}
M_k \eqdef  \left(\prod_{j=0}^{k-1} \lambda_j \right)^{-1}  \; \prod_{j=0}^{k-1} B_j^{-\1_{\bar{C}_j}(X_j,X'_j)} \; V_{k}(X_k,X'_k) \1 \{d_k = 0\} \eqsp.
\end{equation}
Since, by construction,
\begin{multline*}
\CPE[\Xinit,\Xinit',0]{V_{k+1}(X_{k+1},X'_{k+1}) \1 \{d_{k+1}=0\}}{\mathcal{F}_k} \\
(1-\eps_k) \bar{R}_k V_k(X_k,X'_k) \1_{\bar{C}^c_k}(X_k,X'_k) + \lambda_k V_k(X_k,X'_k) \1_{\bar{C}_k}(X_k,X'_k) \eqsp,
\end{multline*}
it is easily shown that $( M_k, k \geq 0)$ is a $(\mathcal{F},\PP^Y_{\Xinit,\Xinit',0})$-supermartingale w.r.t.\
where $ \mcf \eqdef ( \mcf_{k} )_{1 \leq k \leq n}$ with for $k \geq 0$,
$\mcf_k \eqdef \sigma\left[ (X_j,X'_j,d_j) , 0 \leq j \leq k  \right]$. Therefore,
\[
\PE^Y_{\Xinit,\Xinit',0}(M_n) \leq \PE^Y_{\Xinit,\Xinit',0}(M_0) = \Xinit \otimes \Xinit' (V_0) \eqsp.
\]
This establishes \eqref{eq:coupling-drift-1} and concludes the proof.

\end{proof}

\section{Proofs of Section \ref{sec:FAR}}
To simplify the notations, the dependence of $C(y)$ in $K$ is implicit throughout the section.
\label{sec:proof:FAR}
\begin{proof}[Proof of Lemma \ref{lem:FAR}]
Consider first the case  $d(x_1,a(C)) \geq M$. For any $z_1 \in a(C)$,
\begin{align*}
&M \leq |x_1 -a(x_0)| \leq |x_1-z_1| + |z_1-a(x_0)| \leq \diam(C) + |x_1-z_1| \eqsp, \\
&M \leq |x_1-z_1| \leq |x_1-a(x_0)| + |z_1-a(x_0)| \leq \diam(C) + |x_1-a(x_0)| \eqsp.
\end{align*}
Using that $p_U$ is non-increasing for $u \geq M$ and \eqref{eq:condition-noise}, we obtain
\[
p_U(|x_1-a(x_0)|) \geq p_U(\diam(C) + |x_1-z_1|) \geq \gamma p_U(\diam(C)) p_U(|x_1-z_1|) \eqsp,
\]
and similarly,
\[
p_U(|x_1-z_1|) \geq \gamma p_U(\diam(C)) p_U(|x_1-a(x_0)|) \eqsp,
\]
which establishes that \eqref{eq:strong-small-set-FAR} holds when $d(x_0,a(C)) \geq M$.

Consider now the case $d(x_1,a(C)) \leq M$. Since $x_0$ belongs to $C$, then $|x_1-a(x_0)| \leq M + \diam(C)$, which implies that
$$
\inf_{u \leq M + \diam(C)} p_U(u) \leq p_U(|x_1-a(x_0)|) \leq \sup_{u \leq M + \diam(C)} p_U(u) \eqsp,
$$
\eqref{eq:strong-small-set-FAR} holds  for $d(x_1,a(C)) \leq M$.

Consider now the second assertion.
Assume first that $x_0$ is such that $d(a(x_0),C) \geq M$ and let $z_1$ be an arbitrary point of $C$. Then, for any $x_1 \in C$,
\[
M \leq |x_1- a(x_0)| \leq |x_1 - z_1| + |z_1-a(x_0)| \leq \diam(C) + |z_1-a(x_0)| \eqsp.
\]
Using that $p_U$ is monotone decreasing on $[M,\infty)$ and \eqref{eq:condition-noise},
\begin{multline}
\label{eq:a(x0)farfromC}
p_U(|x_1-a(x_0)|) \geq p_U(\diam(C) + |z_1-a(x_0)|) \\
\geq \gamma p_U[\diam(C)] p_U(|z_1-a(x_0)|) \eqsp.
\end{multline}
If $d(a(x_0),C) \leq M$, then for any $x_1 \in C$, $|x_1-a(x_0)| \leq \diam(C) + M$, so that
\begin{equation}
\label{eq:a(x0)closedfromC}
\inf_{|u| \leq \diam(C)+M} p_U  \leq p_U\left(|x_1-a(x_0)|\right) \eqsp.
\end{equation}
\end{proof}

\begin{proof}[Proof of Lemma \ref{lem:borne-inf-uniforme}]
Choose $K$ such that $b_1^{-1} K \geq M$. If $|b^{-1}(y)-x| \geq K$, then,
\begin{equation}
\label{eq:util:FAR:item1}
|y-b(x)| = |b (b^{-1}(y)) - b(x)| \geq b_1^{-1}|b^{-1}(y)-x| \geq M \eqsp,
\end{equation}
and since $p_V$ is non-increasing on the interval $[M,\infty[$, the following inequality holds
\begin{multline*}
\int_{|x-b^{-1}(y)| \geq K} p_V(|y-b(x)|) \rmd x \leq \int_{|x-b^{-1}(y)|\geq M} p_V(b_1^{-1}|b^{-1}(y)-x|) \rmd x \\
\leq b_1 \int_{b_1^{-1} K}^\infty p_V( x) \rmd x \eqsp.
\end{multline*}
Since the Jacobian of $b$ is bounded, $\int p_V(|y-b(x)|) \rmd x$ is bounded away from zero by change of variables. The proof follows.

\end{proof}

\begin{proof}[Proof of Lemma \ref{lem:util:FAR}]
We will establish the results by considering independently the following cases:
\begin{enumerate}
\item \label{lem:util:FAR:item1}
For any $y$ and any $(x_0,x_2)$ such that  $d(a(x_0),C(y)) \leq M$ and  $d(x_2,a[C(y)]) \leq M$,
\[
I(x_0,x_2;y) \leq \left(\sup_{\rset} p_U\right)^2  \eqsp.
\]
\item \label{lem:util:FAR:item2}
For any $y$ and any $(x_0,x_2)$ such that $d(a(x_0),C(y)) > M$ and  $d(x_2,a[C(y)]) \leq M$,
\[
I(x_0,x_2;y) \leq \gamma^{-1} \; \left(\sup p_U\right)   \int_K^\infty [p_U(x)]^{-1} p_V(b_- x) dx \eqsp.
\]
\item \label{lem:util:FAR:item3}
For any $y$ and any $(x_0,x_2)$ such that $d(a(x_0),C(y)) \leq M$ and  $d(x_2,a[C(y)]) > M$
\[
I(x_0,x_2;y) \leq \gamma^{-1} \left( \sup p_U \right) \left\{ b_-^{-1} +   \int_{K}^\infty p_V(b_-  x) [p_U(a_+ x)]^{-1} \rmd x \right\}
\]
\item \label{lem:util:FAR:item4}
For any $y$ and any $(x_0,x_2)$ such that $d(a(x_0),C(y)) > M$ and  $d(x_2,a[C(y)]) > M$,
\begin{multline*}
I(x_0,x_2;y) \leq \gamma^{-2}   \\ \times
\int_{K}^{\infty} [p_U(x)]^{-1} p_V(b_- x) \left\{ \left( \inf_{u \leq M} p_U(u)  \right)^{-1}  +   [p_U(a_+ x)]^{-1} \right\}
\rmd x \eqsp.
\end{multline*}
\end{enumerate}
\begin{proof}[Proof of Assertion \ref{lem:util:FAR:item1}]
On the set $\{x_0, d(a(x_0),C(y)) \leq M \}$, $k_{C(y)}(x_0) \equiv 1$;
On the set  $\{ x_2, d(x_2,a[C(y)]) \leq M\}$, $h_{C(y)}(x_2) \equiv 1$. Since $p_U$ is uniformly bounded,
the bound follows from Lemma \ref{lem:borne-inf-uniforme} and the choice of $K$.
\end{proof}
\begin{proof}[Proof of Assertion \ref{lem:util:FAR:item2}]
On the set $\{x_0, d(a(x_0),C(y)) > M\}$,  $k_C(x_0) = p_U(|b^{-1}(y) -a(x_0)|)$ ;
On the set $\{x_2, d(x_2,a[C(y)]) \leq M\}$, $h_C(x_2) \equiv 1$.
Therefore, for such $(x_0,x_2)$,
\begin{multline}
\label{eq:boundI2}
I(x_0,x_2;y) \leq \left( \sup p_U \right) \\ p_U^{-1}(|b^{-1}(y) - a(x_0)|) \int_{C^c(y)} p_U(|x_1-a(x_0)|) p_V(|y_1-b(x_1)|) \rmd x_1 \eqsp.
\end{multline}
We set $\alpha= x_1- a(x_0)$ and $\beta= b^{-1}(y) -x_1$. Note that $|\alpha + \beta|= |b^{-1}(y)-a(x_0)| \geq d(a(x_0),C(y)) > M$.
Since $p_U$ is non-increasing on $[M,\infty[$, $p_U(|\alpha+\beta|) \geq p_U(|\alpha|+|\beta|)$, and
the condition \eqref{eq:condition-noise} shows that $\left(p_U(|\alpha+\beta|)\right)^{-1} p_U(|\alpha|) \leq \gamma^{-1} p_U^{-1}(|\beta|)$
which implies
\begin{equation}
\label{eq:util:FAR:item2:cond}
p_U^{-1}(|b^{-1}(y)-a(x_0)|) p_U(|x_1-a(x_0)|) \leq \gamma^{-1} p_U^{-1}(|b^{-1}(y)-x_1|) \eqsp.
\end{equation}
Therefore, plugging \eqref{eq:util:FAR:item2:cond} into the RHS of \eqref{eq:boundI2} yields
\begin{align*}
I(x_0,x_2;y) &\leq \gamma^{-1} \left( \sup p_U \right) \int_{|x_1-b^{-1}(y)|\geq K} p_U^{-1}(|b^{-1}(y)-x_1|) p_V(b_-|b^{-1}(y)-x|) \rmd x_1 \\
&\leq \gamma^{-1} \left( \sup p_U \right)  \int_K^\infty p_U^{-1}(x) p_V(b_- x) \rmd x  \eqsp.
\end{align*}

\end{proof}
\begin{proof}[Proof of Assertion \ref{lem:util:FAR:item3}]
On the set $\{x_0, d(a(x_0),C(y)) \leq M\}$,  $k_C(x_0) \equiv 1$;
on the set $\{x_2, d(x_2,a[C(y)]) > M\}$, $h_C(x_2)= p_U(|x_2-a[b^{-1}(y)]|) \equiv 1$.
Therefore, for such $(x_0,x_2)$;
\begin{multline}
\label{eq:boundI3}
I(x_0,x_2;y) \leq \left( \sup p_U \right) \\ \times
p_U^{-1}(|x_2-a[b^{-1}(y)]|) \int_{C^c(y)} p_V(|y-b(x_1)|) p_U(|x_2-a(x_1)|) \rmd x_1 \eqsp.
\end{multline}
We set $\alpha= x_2- a(x_1)$, $\beta= a(x_1) - a[b^{-1}(y)]$. Since $|\alpha+\beta| \geq d(x_2,a[C(y)]) > M$, using as above that
$\left(p_U(|\alpha+\beta|)\right)^{-1} p_U(|\alpha|) \leq \gamma^{-1} p_U^{-1}(|\beta|)$, we show
\begin{equation}
\label{eq:util:FAR:item3:cond}
p_U^{-1}(|x_2-a[b^{-1}(y)]|) p_U(|x_2-a(x_1)|) \leq \gamma^{-1} p_U^{-1}(|a(x_1) - a[b^{-1}(y)]|) \eqsp.
\end{equation}
Since for any $x,x' \in \Xset$,
\begin{multline}
\label{eq:util:FAR:item3:cond1}
p_U^{-1}(|a(x) - a(x')|) \leq \left(\inf_{u \leq M} p_U(u) \right)^{-1} \1 \{ |a(x) - a(x')| \leq M \} \\
+  p_U^{-1}(a_+|x-x'|) \1 \{ |a(x) -a(x')| > M \} \eqsp,
\end{multline}
the RHS of \eqref{eq:boundI3} is therefore bounded by
\begin{multline*}
I(x_0,x_2;y) \leq \gamma^{-1} \left( \sup p_U \right) \int_{|x_1-b^{-1}(y)| \geq K} p_V(b_-(|x_1 -b^{-1}(y)|)) \\
\left\{ \left( \inf_{u \leq M} p_U(u)  \right)^{-1}  + p_U^{-1}(a_+ |x_1-b^{-1}(y)|) \right\} \rmd x_1 \eqsp.
\end{multline*}
\end{proof}
\begin{proof}[Proof of Assertion \ref{lem:util:FAR:item4}]
On the set $\{x_0, d(a(x_0),C(y)) > M \}$, $k_{C(y)}(x_0) = p_U(|b^{-1}(y)-a(x_0)|)$. On the
set $\{x_2, d(x_2,a[C(y)]) > M \}$, $k_{C(y)}(x_2) = p_U(|x_2-a[b^{-1}(y)])$. Therefore, for such $(x_0,x_2)$,
\begin{multline}
\label{eq:boundI4}
I(x_0,x_2;y) \leq p_U^{-1}( |b^{-1}(y)-a(x_0)|) p_U^{-1}(|x_2-a[b^{-1}(y)]|) \\
\times \int_{C^c(y)} p_U(|x_1-a(x_0)|) p_V(|y-b(x_1)|) p_U(|x_2-a(x_1)|) \rmd x_1 \eqsp.
\end{multline}
Using \eqref{eq:util:FAR:item1}, \eqref{eq:util:FAR:item2:cond}, \eqref{eq:util:FAR:item3:cond}, and \eqref{eq:util:FAR:item3:cond1},
the RHS of the previous equation is bounded by
\[
I(x_0,x_2; y) \leq \gamma^{-2} \int_K^{\infty} p_U^{-1}(|x|) p_V(b_-|x|) \left\{ \left( \inf_{u \leq M} p_U(u)  \right)^{-1}  + p_U^{-1}(a_+ x) \right\}
\rmd x \eqsp.
\]
The proof follows.
\end{proof}
\end{proof}

%


\end{document}